\documentclass[12pt,a4paper]{amsart}

\usepackage{pgfplots}
\pgfplotsset{compat=1.15}
\usepackage{mathrsfs}
\usetikzlibrary{arrows}

\usepackage{xcolor}
\usepackage[colorlinks=true,linkcolor=blue,citecolor=red]{hyperref}

\makeatletter
\long\def\unmarkedfootnote#1{{\long\def\@makefntext##1{##1}\footnotetext{#1}}}
\makeatother

\usepackage{amsmath,amssymb}
\usepackage{enumerate,amssymb}
\usepackage{color}
\theoremstyle{plain}
\newtheorem{thm}{Theorem}[section]

\newtheorem{lemma}[thm]{Lemma}

\newtoks\prt
\newtheorem{proclaim}[thm]{\the\prt}
\theoremstyle{definition}

\def\eqn#1$$#2$${\begin{equation}\label#1#2\end{equation}}

\numberwithin{equation}{section}

\def\epsilon{\varepsilon}

\def\er{\mathbb R}

\def\mir1{\mathcal L_1}

\def\oint{-\hskip -11pt \int}

\def\Q{\tilde{Q}}

\def\Q{\widetilde{Q}}

\def\rn{\mathbb R^n}

\def\haus{\mathcal{H}}

\newcommand{\RR}{\mathbb{R}}

\newcommand{\NN}{\mathbb{N}}

\newcommand{\I}{\int_}
\newcommand{\Om}{\Omega}

\newtoks\by
\newtoks\paper
\newtoks\book
\newtoks\jour
\newtoks\yr
\newtoks\pages
\newtoks\vol
\newtoks\publ
\def\ota{{\hbox\vol{???}}}
\def\cLear{\by=\ota\paper=\ota\book=\ota\jour=\ota\yr=\ota
\pages=\ota\vol=\ota\publ=\ota}
\def\endpaper{\the\by, {\the\paper},
\textit{\the\jour} \textbf{\the\vol} (\the\yr), \the\pages.\cLear}
\def\endbook{\the\by, \textit{\the\book}, \the\publ.\cLear}
\def\endprep{\the\by, \textit{\the\paper}, \the\jour.\cLear}
\def\endyearprep{\the\by, \textit{\the\paper}, \the\jour, (\the\yr).\cLear}
\def\name#1#2{#2 #1}
\def\nom{ \rm no. }

\title{Note on injectivity in second-gradient Nonlinear Elasticity}

\author{Stanislav Hencl}
\address{Department of Mathematical Analysis, Charles University,
	So\-ko\-lovsk\'a 83, 186~00 Prague 8, Czech Republic}
\email{hencl@karlin.mff.cuni.cz}

\author{Kaushik Mohanta}

\address{Department of Mathematical Analysis, Charles University,
	So\-ko\-lovsk\'a 83, 186~00 Prague 8, Czech Republic}
\email{kaushik.mohanta@matfyz.cuni.cz}

\date{\today}

\thanks{The authors were supported by the grant GA\v{C}R P201/24-10505S}

\begin{document}

\begin{abstract}
    {Let $q>1$, $(1-\frac{1}{q})a\geq 1$ and let $\Omega\subset \mathbb{R}^2$ be Lipschitz domain. We show that planar mappings in the second order Sobolev space $f\in W^{2,q}(\Omega,\mathbb{R}^2)$ with $|J_f|^{-a}\in L^1(\Omega)$ are homeomorphism if they agree with a homeomorphism on the boundary. The condition $(1-\frac{1}{q})a\geq 1$ is sharp. We also have a new sharp result about the $\mathcal{H}^{n-1}$ measure of the projection of the set $\{J_f=0\}$ in $\mathbb{R}^n$.}
\end{abstract}

\maketitle

\section{Introduction}

We study models of Nonlinear Elasticity that involve the second gradient, and in particular, we would like to study injectivity of such mappings which corresponds to the physical assumption of ``noninterpenetration of the matter''. 
	Let $\Omega\subset\rn$ be a domain and let $f:\Omega\to\rn$. 
In this paper we study mappings with finite energy (with some proper convex functions $W$ and $\Psi$)
$$
E(f):=
\int_{\Omega} \bigl(W(Df(x))+\Psi(D^2 f(x))\bigr)\; dx
$$
such that there are $q\geq 1$ and $a>0$ so that  
\eqn{assume}
$$
W(Df)\geq \frac{1}{|J_f|^a}\text{ and }\Psi(D^2 f)\geq |D^2f|^q, 
$$
i.e. $f\in W^{2,q}$ and $J_f^{-a}\in L^1$. Moreover, we assume that the mapping $f$ is equal to a given homeomorphism $f_0$ on $\partial \Omega$ so we prescribe our deformation there. 
	 Models with the second gradient were introduced by Toupin \cite{T}, \cite{T2}
	and later considered by many other authors, see e.g. Ball, Curie, Olver \cite{BCO}, Ball, Mora-Corral \cite{BMC}, M\"uller \cite[Section 6]{Mbook}, Ciarlet \cite[page 93]{Ci} and references given there 
	(for references to some newer results see e.g. introduction in \cite{CHMS}). 
	The contribution of the higher gradient is usually connected with interfacial energies and is used to model various phenomena like elastoplasticity or damage and the term $1/|J_f|^a$ corresponds to the penalisation of compression.

Our main result is the following. 

\prt{Theorem}
\begin{proclaim}
\label{main}
Let $n=2$, $q> 1$ and $(1-\frac{1}{q})a\geq 1$. 
Let $\Omega\subset\er^2$ be a Lipschitz domain and let $f_0:\overline{\Omega}\to\er^2$ be a given homeomorphism. 
Assume that $f\in W^{2,q}(\Omega,\er^2)$ is a mapping such that $|J_f|^{-a}\in L^1(\Omega)$ and $f=f_0$ on $\partial \Omega$ in the sense of traces. 
Then $f$ is a homeomorphism in $\Omega$. 
\end{proclaim}

The above theorem was originaly shown in \cite{HK} only for $q>2$ under the very restrictive assumption $(1-\frac{n}{q})a\geq n$. Then it was generalized also to $q>4/3$ in \cite{CHMS} under the condition $(\frac{3}{2}-\frac{2}{q})a\geq 1$ which is better than \cite{HK}. We are able to show the result in full generality $q> 1$ under the weakest condition.  
Moreover, the simple ``folding'' example in \cite[Example 5.1]{CHMS} shows that our condition is sharp as there are counterexamples for $(1-\frac{1}{q})a< 1$. 
Let us note that \cite{HK} and \cite[Theorem 1.2]{CHMS} contain some result also for $n\geq 3$ but our approach does not improve those.

The main tool for the proof is the following estimate of the measure of the set where the Jacobian can vanish. 
In \cite{HK} it was shown that the set $\{J_f=0\}$ is empty. Our estimate on the dimension of $\{J_f=0\}$ was shown in \cite[Theorem 1.3 and Corollary 1.4]{CHMS} only in the case $q>n$ and in the case 
$\frac{n^2}{2n-1}<q<n$ they needed much stricter assumption. Again we have the result in full generality $q> 1$ and our assumption $(1-\frac{1}{q})a\geq 1$ is sharp even in this case as shown by \cite[Example 5.1]{CHMS}. 


\prt{Theorem}
\begin{proclaim}\label{technical}
Let  $n\geq 2$, $\Omega\subset\er^n$ be a domain, $q> 1$ 
and let $(1-\frac{1}{q})a\geq 1$. 
Assume that $f\in W^{2,q}(\Omega,\er^n)$ is a mapping with $|J_f|^{-a}\in L^1(\Omega)$and let $\pi$ denote the projection of $\rn$ to some hyperplane. Then the precise representative of $Df$ satisfies $\haus^{n-1}\bigl(\pi(\{J_f=0\})\bigr)=0$. 

It follows that either $J_f>0$ a.e. in $\Omega$ or $J_f<0$ a.e. in $\Omega$. 

\end{proclaim}

In the theorem above we use the standard representative $Df(x)=\limsup_{r\to 0+}\oint_{B(x,r)} Df$ and $J_f(x)$ is the determinant of this precise representative.


The main difference with the proof in \cite{CHMS} is the following. Since the Jacobian is essentially a product of $n$ first order derivatives we can estimate its derivative by product rule as  
$$
|DJ_f|\leq C |D^2f| |Df|^{n-1},
$$
Since $f\in W^{2,q}$ we know that for $q>n$ we have $|Df|\in L^{\infty}$ and hence it is possible to show that $J_f\in W^{1,q}$ by the above inequality but for $q<n$ we get only a weaker result (see \cite[Lemma 2.2]{CHMS} for details). In our proof for $q<n$ we know that $|Df|$ is not bounded a.e. but it is in fact bounded on a set which is big enough, so we can use similar tools to \cite{CHMS}-like analogue of Poincar\'e inequality on a big enough set and clever coverings to make the proof work. 

\bigskip
 \section{Proof of the Theorems}

    \begin{lemma}\label{lm-covering}
  Let $\alpha>2$ and let $E\subset \RR^n$ satisfy $0<|E|<\infty$. Then for every $\varepsilon>0$, there exists $\{Q_i\}_i$, a collection of closed cubes in $\RR^n$, such that
	\eqn{aha}
  $$
  \text{diam}(Q_i) <\varepsilon, E\subseteq \bigcup_{i=1}^\infty Q_i \text{ and }\sum_i|Q_i|< 2|E|. 
  $$
	Moreover, for the the set
	$$
  \Gamma:=\bigl\{i\in \NN\ |\ |Q_i|\leq \alpha |Q_i\cap E| \bigr\}\text{ we have }(1-\frac{2}{\alpha}) |E|\leq \sum_{i\in\Gamma} |Q_i|.
   $$
\end{lemma}
\begin{proof}
We can find an open set $G\subset\rn$ with $E\subset G$ and $|G|< 2|E|$. We can cover this $G$ by disjoint dyadic cubes (boundaries of cubes can intersect) so that we have \eqref{aha}. Simple estimate gives us
$$
\sum_{i\notin \Gamma}|Q_i\cap E|\leq \sum_{i\notin \Gamma}\frac{1}{\alpha}|Q_i|\leq \frac{1}{\alpha}\sum_{i}|Q_i|<\frac{2}{\alpha} |E|. 
$$
It follows that 
$$
|E|= \sum_{i\in \Gamma}|Q_i\cap E|+\sum_{i\notin \Gamma}|Q_i\cap E|\leq \sum_{i\in \Gamma}|Q_i|+\frac{2}{\alpha}|E|. 
$$
and our conlusion follows easily. 
\end{proof}

We need the well-known ACL characterization of Sobolev spaces (see e.g. \cite[Chapter~4.9]{EG} or \cite[Theorem A.15]{HK})

\prt{Theorem}
\begin{proclaim}\label{ACL}
Let $\Omega\subset \rn $ be an open set and let $p\in [1,\infty)$. Then $h\in W^{1,p}(\Omega)$ if and only if it satisfies the ACL condition, i.e. it has a representative that is absolutely continuous on almost all line segments parallel to coordinate axes and moreover, the partial derivatives of these absolutely continuous functions belong to $L^{p}(\Omega)$.   
\end{proclaim}

\begin{lemma}\label{lm-Poinc}
Let $\pi:(x_1,\cdots,x_n)\mapsto (x_1,\cdots,x_{n-1})$ be the standard projection and let $\tilde{E}\subset (-1,1)^{n-1}$ satisfy $0<\haus^{n-1}(\tilde{E})<\infty$. 
Assume that $g:(-1,1)^n\to\er$ satisfies the ACL condition and that $\pi\left(g^{-1}(0)\right)=\tilde{E}$. 
Let $E\subset (-1,1)^n$ be a measurable set so that $\pi(E)=\tilde{E}$ and for each $x'\in \tilde{E}$ we know that $\pi^{-1}(x')\cap E$ is an interval $I_x$ of sidelength $\ell(I)$ and that 
$I_{x'}\cap \{g^{-1}(0)\}\neq 0$.  
Then we have	
$$
\I{E} |g|\, dx 
\leq \ell(I) |E|^{1-\frac{1}{q}} \left(\I{E}|D g|^q\, dx\right)^{\frac{1}{q}}.
$$
\end{lemma}
\begin{proof}
As $g$ satisfies the ACL condition, it is absolutely continuous on almost all segments and thus $Dg$ is well defined a.e. If $\int_E|Dg|=\infty$, then there is nothing to prove. Hence, for the rest of the proof, we assume that $\int_E|Dg|<\infty$. For each $x'\in \tilde{E}$, consider the vertical line-segment $L(x'):=I_{x'}$. By the ACL condition on $g$, we get, for a.e. $x'\in \tilde{E}$, and for any $t_1,t_2\in I_{x'}$,
 $$
 |g(x',t_1)-g(x',t_2)| \leq \I{L(x')}|D g(x',t)| d\haus^1(t).
 $$
 From the hypothesis $I_{x'}\cap \{g^{-1}(0)\}\neq 0$, we know that we may choose $t_2\in I_{x'}$ such that $g(x',t_2)=0$. We choose such a $t_2$ and integrate over $x'\in \tilde{E}$ and over 
$t_1\in I_{x'}$ to get
 $$
 \I{x'\in \tilde{E}}\I{t_1\in I_{x'}}|g(x',t_1)|d\haus^1(t_1) d\haus^{n-1}(x') 
 \leq \I{x'\in \tilde{E}} \I{t_1\in I_{x'}} \I{I_x'}|D g(x',t)| d\haus^1(t) d\haus^1(t_1) d\haus^{n-1}(x').
 $$ 
 Now, Fubini's theorem, followed by H\"older's inequality gives
$$
\I{x\in E}|g(x)|dx 
\leq \ell(I) |E|^{1-\frac{1}{q}} \left(\I{x\in E} |D g(x)|^q dx\right)^\frac{1}{q}.
$$
\end{proof}

   
    \begin{proof}[\textbf{Proof of Theorem \ref{technical}}]
		This statement was shown in \cite[Theorem 1.3]{CHMS} in the case $q>n$ so we may assume that $q\leq n$ here. The fact that $\haus^{n-1}\bigl(\pi(\{J_f=0\})\bigr)=0$ implies that $J_f$ does not change sign is the same as in \cite[proof of Corollary 1.4]{CHMS}. 
		
         
        Since the result is local, we may, without loss of generality, assume that $\Om=(-1,1)^n$, $\tilde{H}=(-1,1)^{n-1}$, and 
        $$
        \pi:\Om\ni (x',x_n)\mapsto x'\in\tilde{H}.
        $$ 
				    For $t\in (-1,1)$, denote $H_t:=\tilde{H}\times \{t\}$ and $H=H_0$. Since $f\in W^{2,q}(\Om)$, we have $\I{\Om} |D^2 f|^q dx<\infty$. 
				It is not difficult to show using Fubini's theorem and ACL condition (see Theorem \ref{ACL}) that, for almost all $t\in (-1,1)$, $f\in W^{2,q}(H_t\cap \Om)$. Without loss of generality, we may assume that $t=0$ is such a choice. So, we have
				\eqn{finite}
        $$
    f\in W^{2,q}(H)\text{ and }\I{H}|D^2f|^q d\mathcal{H}^{n-1}(x)<\infty.
        $$

				We use the standard representative 
				$$
				Df(x)=\limsup_{r\to 0+}\oint_{B(x,r)} Df
				$$
				for the $Df$ which is Borel measurable (as limsup of continuous functions). It follows that our representative of $J_f(x)$ which is the determinant of this representative is also Borel measurable.   
				It follows that the set $Z:= \{x\in \Om\ |\ J_f(x)=0\}$ is Borel and its projection is measurable (it is even analytic, i.e. continuous image of a Polish space, see \cite{Ke}). 
        We assume for contradiction that 
        	\eqn{eq6} 
        	$$
           \haus^{n-1}(\pi(Z))>0.
            $$
						Now we find a good set $G\subset (-1,1)^{n-1}$ with $\haus^{n-1}(\pi(Z)\cap G)>\tfrac{\haus^{n-1}(\pi(Z))}{2}$ so that $|Df|$ is not too big on $\pi^{-1}(G)$. 						
						
						 Using \eqref{finite} we can find $\lambda_1$ big enough so that for the set 
				\eqn{eq3}
				$$
				F:= \{x'\in (-1,1)^{n-1}\ |\ |Df(x',0)|>\lambda_1 \}\text{ we have }\mathcal{H}^{n-1}\left( F \right)<\frac{\haus^{n-1}(\pi(Z))}{4}.
				$$
			   Since, $f\in W^{2,q}(\Om)$, we can choose $\lambda_2=\lambda_2(\lambda_1)>0$ large enough, so that 
        \eqn{eq4}
        $$
        (\lambda_2-\lambda_1)\frac{\haus^{n-1}(\pi(Z))}{4} > \I{\Om} |D^2f|.
        $$
	Denote 
	\eqn{defG}
	$$
	A:=\{x\in \Om \ |\ |Df|\geq \lambda_2\}\text{ and }G:= (-1,1)^{n-1}\setminus (F\cup \pi (A)). 
	$$
	For any $x'\in \pi(A)\setminus F$ we know that $|Df(x',0)|<\lambda_1$ and we can find $t_1\in(-1,1)$ with 
	$$
  |Df(x',t_1)|>\lambda_2.
  $$
          Now, from the ACL condition (see Theorem \ref{ACL}), for $\haus^{n-1}$-a.e. $x'\in H$, $|Df|$ is absolutely continuous on $L(x'):=\pi^{-1}(x')$ and hence for those $x'\in \pi(A)\setminus F$
        $$
        (\lambda_2-\lambda_1)
		\leq (|Df(y)|-|Df(x)|)
        \leq \I{L(x')} |D^2f(x',t)| d\mathcal{H}^{1}(t).
        $$
        We integrate for $x'\in \pi(A)\setminus F$ and use Fubini's theorem to get 
      $$
	\mathcal{H}^{n-1}   \left(\pi(A)\setminus F\right)(\lambda_2-\lambda_1)
        \leq \I{\pi(A)\setminus F} \I{L(x')}|D^2f(x',t)|d\mathcal{H}^{1}(t)d\mathcal{H}^{n-1}(x')
        \leq \I{\Om} |D^2f(x)|dx.
      $$
        Combining this with \eqref{eq4}, \eqref{eq3} and \eqref{defG}, we get 
        $$
        \mathcal{H}^{n-1}(\pi(A)\setminus F)
        \leq \frac{|\pi(Z)|}{4}\text{ and hence }\mathcal{H}^{n-1}(G\cap \pi(Z))>\frac{\haus^{n-1}(\pi(Z))}{2}. 
        $$

        Let us choose $\varepsilon>0$ arbitrarily. Then, 
        we can get countably many disjoint closed cubes $\{\Q_i\}_{i\in \NN}$ 
				such that
	$$
	\pi(Z)\cap G\subseteq \bigcup_{i=1}^\infty \Q_i\subseteq (-1,1)^{n-1},
	$$
	and  
	\eqn{choosewell}
        $$
        \mathcal{H}^{n-1}(\pi(Z)\cap G)\leq  \sum_{i=1}^\infty \ell(\Q_i)^{n-1}
        \leq 2 \mathcal{H}^{n-1}(\pi(Z)\cap G). 
        $$
				For each $\Q_i$ we now find a measurable set $Q_i\subset (-1,1)^n$ so that $\pi(Q_i)=\Q_i$ (see Fig. \ref{projection}),
				$$ 
				\text{ for each }x'\in \Q_i \quad \pi^{-1}(x')\cap Q_i\text{ is an interval }I_{x'}\text{ of length }
				\ell(\Q_i) 
				$$
				$$
				\text{ and for every }x'\in \Q_i \cap (\pi(Z)\cap G)\text{ we know that }I_{x'}\cap Z\neq \emptyset.
				$$	
 Since $\pi(Z)$ is analytical, the selection theorem \cite[Theorem 18.1]{Ke} allows us, for each $x'\in \pi(Z)$,  to choose $(x',x_n)\in Z$ in a measurable way. This implies that the sets $Q_i$ can be assumed to be measurable.
				
  

	\begin{figure}
\phantom{a}
\vskip 170pt
{\begin{picture}(0.0,0.0) 
     \put(-230.2,0.2){\includegraphics[width=1.00\textwidth]{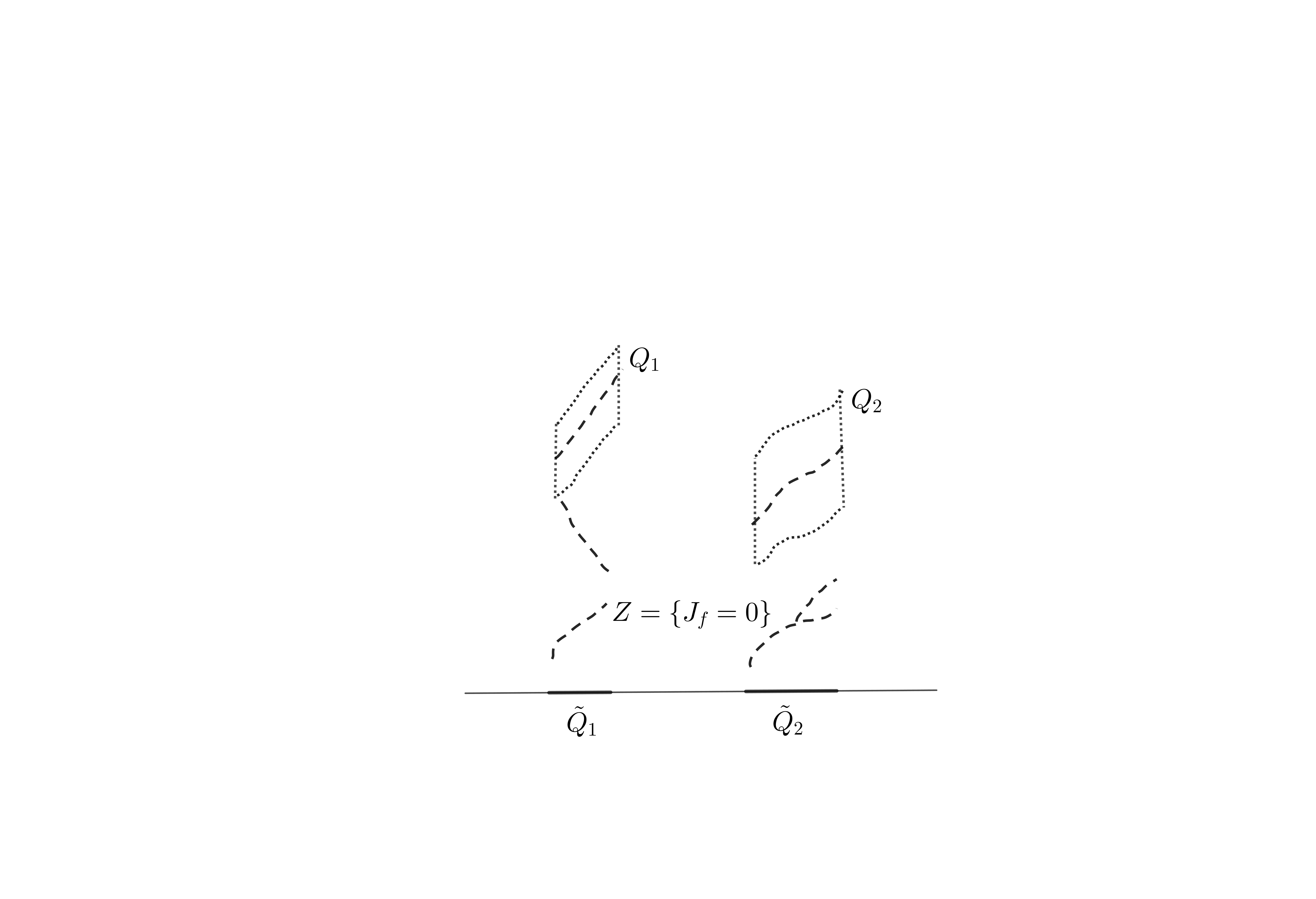}}
  \end{picture}
  }
\vskip -55pt	
\caption{Definition of $Q_i$. Given cubes $\Q_1$ and $\Q_2 $ we find sets $Q_1$ and $Q_2$ (denoted by dotted curves) that project onto $\Q_1$ and $\Q_2$ and intersect $Z=\{J_f=0\}$ (denoted by dashed curves) in each vertical segment if possible.}\label{projection}
\end{figure}

      We may assume that $\{Q_i\}_i$ have a small total measure and thus by the absolute continuity of the integral
               \eqn{abscont}
               $$
       	\sum_{i=1}^\infty \int_{Q_i}|J_f|^{-a}\; dx <\epsilon. 
               $$

	In view of \eqref{choosewell} and Lemma \ref{lm-covering} for $\alpha=3$ and $E=\pi(Z)\cap G$ we get
	\eqn{eq-LB-1}
	$$
	\haus^{n-1}(\pi(Z)\cap G)\leq 3\sum_{i\in \Gamma} \haus^{n-1}(\Q_i),
	$$
	where 
	$$
	\Gamma=\bigl\{i: \haus^{n-1}(\Q_i) \leq 3 \haus^{n-1}(\Q_i\cap \pi(Z)\cap G)\bigr\}.
	$$
	Consider, for $i\in \NN$,
            $$
            E_i:=\pi^{-1}\bigl(\Q_i\cap \pi\left(Z\right)\cap G\bigr)\cap Q_i.
            $$
	    From the definition of $\Gamma$ for 
      \begin{equation}\label{cstar}
      \quad i\in \Gamma\text{ we have } \haus^{n-1}(\Q_i)\leq 3 \haus^{n-1}(\pi(E_i)) \text{ and hence }|Q_i|\leq 3|E_i| .
      \end{equation}

			Since the Jacobian is essentially a product of $n$ first order derivatives we can estimate its derivative by product rule as  
       (see \cite[Chapter~4,9]{EG} for details)
        $$
        |DJ_f(x)| \leq C |D^2f(x)|\ |Df(x)|^{n-1}\text{ for a.e. }x\in \Omega.
        $$
			For any $x\in E_i$ we know that $\pi(x)\in G$ and hence $|Df(x)|\leq \lambda_2$ by \eqref{defG}. 
	This, along with the above inequality implies that
	$$
	|DJ_f(x)| \leq C \lambda_2^{n-1}\ |D^2f(x)|.
	$$ 
	Raising this to the power $q$ and then integrating over $E_i$, we get, for all $i\in \NN$,
        \eqn{mfd}
        $$
        \I{E_i}|DJ_f(x)|^q dx 
        \leq C \lambda_2^{(n-1)q}\I{E_i}|D^2f(x)|^q dx.
        $$
				Since $J_f$ is a product of $n$ derivatives and these derivatives are absolutely continuous by Theorem \ref{ACL} it is not difficult to see that $J_f$ satisfies the ACL condition as well 
				(in fact our precise representative of $Df$ satisfies the ACL and thus our representative of $J_f$ satisfies it as well). 
         For all $i\in \NN$, Lemma \ref{lm-Poinc} implies that
        $$
	\I{E_i} |J_f|\, dx 
	\leq 2\ell(Q_i) |E_i|^{1-\frac{1}{q}} \left(\I{E_i}|DJ_f|^q\, dx\right)^{\frac{1}{q}}.
        $$
        Combining this with \eqref{mfd} and $|E_i|\leq |Q_i|$, we get
        \eqn{poinc}
        $$
        \I{E_i} |J_f|\, dx 
        \leq C \lambda_2^{n-1} \ell(Q_i)^{1+n-\frac{n}{q}} \left(\I{Q_i}|D^2f(x)|^q dx\right)^{\frac{1}{q}}.
        $$
        Using \eqref{cstar} and then choosing $p=1+\frac{1}{a}$, H\"older's inequality and \eqref{poinc} give us 
        \begin{align*}
           |Q_i|\leq C |E_i|
      &\leq \I{E_i} |J_f|^\frac{1}{p} |J_f|^\frac{-1}{p} dx\\
            &\leq \Bigl(\I{E_i} |J_f| dx\Bigr)^\frac{1}{p}\Bigl(\I{E_i} |J_f|^{-a} dx\Bigr)^\frac{1}{1+a}\\
            &\leq C \lambda_2^\frac{n-1}{p} \ell(Q_i)^{\frac{1}{p}+\frac{n}{p}-\frac{n}{pq}}\Bigl(\I{Q_i} |D^2f|^q dx\Bigr)^\frac{1}{pq}\Bigl(\I{Q_i} |J_f|^{-a} dx\Bigr)^\frac{1}{1+a}.
        \end{align*}
        Taking sum over all $i\in \Gamma$ 
        \begin{align}\label{1}
		\sum_{i=1}^\infty \I{Q_i} |J_f|^{-a} dx
            > C \lambda_2^{a(1-n)} \sum_{i\in \Gamma} \frac{ \ell(Q_i)^{a(\frac{n}{q}-n-1)} |Q_i|^{1+a}}{\Bigl(\I{Q_i} |D^2f|^q dx\Bigr)^\frac{a}{q}}.
        \end{align}
        Choose $\alpha=\frac{a+q}{q}$. Then $\alpha'=\frac{a+q}{a}$. Let
        $$
        a_i= \Bigl( \frac{ \ell(Q_i)^{a(\frac{n}{q}-n-1)} |Q_i|^{1+a}}{\Bigl(\I{Q_i} |D^2f|^q dx\Bigr)^\frac{a}{q}} \Bigr)^\frac{1}{\alpha},
        \text{ and }
        b_i= \Bigl(\I{Q_i} |D^2f|^q dx\Bigr)^\frac{a}{q\alpha},
        $$
        and apply the H\"older inequality
        $$
        \Bigl(\sum a_i^{\alpha}\Bigr)^{\frac{1}{\alpha}} 
        \geq \frac{\sum a_i b_i}{\Bigl(\sum b_i^{\alpha'} \Bigr)^{\frac{1}{\alpha'}}} 
        $$
        on \eqref{1}, followed by the fact that $\int_{\Om}|Dg|^q\leq C$, to get
				\eqn{eq9}
				$$
        \begin{aligned}
        \left(\sum_{i=1}^\infty \I{Q_i} |J_f|^{-a} dx\right)^\frac{1}{\alpha}
        &\geq C \lambda_2^\frac{aq(1-n)}{a+q} \frac{\sum_{i\in \Gamma} \left(  \ell(Q_i)^{a(\frac{n}{q}-n-1)} |Q_i|^{1+a} \right)^\frac{1}{\alpha}}{\left(\I{\Om} |D^2f|^q dx \right)^\frac{a}{a+q} }\\
        &\geq C \lambda_2^\frac{aq(1-n)}{a+q} \sum_{i\in \Gamma}  \ell(Q_i)^\frac{an-anq-aq}{a+q} |Q_i|^\frac{q+qa}{a+q}\\        
	    &\geq C  \lambda_2^\frac{aq(1-n)}{a+q} \sum_{i\in \Gamma}  \ell(Q_i)^{\frac{an-anq-aq}{a+q}+n \frac{q+qa}{a+q}} \\
            &\geq C \lambda_2^\frac{aq(1-n)}{a+q} \sum_{i\in \Gamma}  \ell(Q_i)^\frac{an-aq+nq}{a+q}. 
         \end{aligned} 
       $$
        Since $(1-\frac{1}{q})a\geq 1$, we have 
        $$
        n-1 \geq \frac{an-aq+nq}{a+q}.
        $$
        Consequently, we get using \eqref{eq-LB-1}
    \begin{equation}\label{eq8}
        \begin{split}
       	\left(\sum_{i=1}^\infty \I{Q_i} |J_f|^{-a} dx\right)^\frac{1}{\alpha}
        &\geq C(n,a,q)  \lambda_2^\frac{aq(1-n)}{a+q} \sum_{i\in \Gamma}  \ell(Q_i)^{n-1}\\
        &\geq C(n,a,q,Z)  \lambda_2^\frac{aq(1-n)}{a+q} \haus^{n-1}(\pi(Z)\cap G).   
    \end{split}
    \end{equation}    
        Note that $\varepsilon>0$ was chosen arbitrarily and hence the above inequality contradicts \eqref{abscont}.
    \end{proof}

    	
    
            \begin{proof}[\textbf{Proof of Theorem \ref{main}}]
    	From Theorem \ref{technical} we know that either $J_f>0$ a.e. or $J_f<0$ a.e. Without loss of generality we assume that $J_f>0$ a.e.. Note that, since $n=2$, we have $q^*>2$. 
    	Since $f\in W^{2,q}$ implies that $f\in W^{1,q^*}\subset W^{1,2}$ and thus $J_f\in L^1$, we obtain that $f$ is a mapping of finite distortion (see e.g. \cite{HK14} for the definition). Moreover, we obtain that $f$ is also continuous. 
    	
    	To obtain our conclusion it is enough to apply \cite[Theorem 2.5]{CHMS} and hence it is enough to check that  $K_f:=\frac{|Df|^2}{J_f}\in L^1$. 
    	For $q>2$ we obtain that $|Df|\in L^{\infty}$ and the condition $a\geq 1$ implies that $K_f\in L^1$. 
    	In the case $1<q<2$ we use H\"older's inequality 
    	$$
    	\int_{\Om} \frac{|Df|^2}{J_f}\leq \Bigl(\int_{\Om} \frac{1}{J_f^a}\Bigr)^{\frac{1}{a}}\Bigl(\int_{\Om} |Df|^{2\frac{a}{a-1}}\Bigr)^{\frac{a-1}{a}}.
    	$$
    	Clearly $(2-\frac{2}{q})a\geq (1-\frac{1}{q})a\geq 1$ and thus 
    	$(2-\frac{2}{q})a\geq 1$ implies that
    	$$
    	2\frac{a}{a-1}\leq q^*=\frac{2q}{2-q}
    	$$
    	and the last integral is finite by $Df\in L^{q*}$. Similarly for $q=2$ we use H\"older's inequality, $a\geq 2$ and $Df\in L^4$. 
    	
    \end{proof}


\begin{thebibliography}{00}

\bibitem{BCO}
\by{\name{Ball}{J.}, \name{Currie}{J.C.}, \name{Olver}{P.J.}}
\paper{Null Lagrangians, weak continuity, and variational problems of arbitrary order}
\jour{J. Funct. Anal.}
\vol{41\nom 3--4}
\pages{135--174}
\yr{1981}
\endpaper

\bibitem{BMC}
\by{\name{Ball}{J.M.}, \name{Mora-Corral}{C.}}
\paper{A variational model allowing both smooth and sharp phase boundaries in solids}
\jour{Commun. Pure Appl. Anal.}
\vol{8\nom 1}
\pages{55--81}
\yr{2009}
\endpaper

\bibitem{CHMS}
\by{\name{D.}{Campbell}, \name{S.}{Hencl}, \name{A.}{Menovchikov}, \name{S.}{Schwarzacher}}
\paper{Injectivity in second-gradient Nonlinear Elasticity}
\jour{arXiv 2204.05559}
\endprep

\bibitem{Ci}
\by{\name{Ciarlet}{P. G.}}
\book{Mathematical Elasticity. Vol. I. Three-dimensional elasticity.}
\publ{Studies in Mathematics and its Applications, {20}. North-Holland Publishing Co., Amsterdam, 1987}
\endbook

\bibitem{EG}
\by{\name{L.C.}{Evans} and \name{R.F.}{Gariepy}}
\book{Measure theory and Fine Properties of Functions}
\publ{Studies in Advanced Mathematics. CRC Press, Boca Raton, FL}
\yr{1992}
\endbook

\bibitem{HK}
\by{\name{Healey}{T.J.}, \name{Kr\"omer}{S.}}
\paper{Injective weak solutions in second-gradient nonlinear elasticity}
\jour{ESAIM Control Optim. Calc. Var.}
\vol{15}
\pages{863--871}
\yr{2009}
\endpaper

\bibitem{HK14}
\by{\name{Hencl}{S.}, \name{Koskela}{P.}}
\book{Lectures on Mappings of Finite Distortion. Lecture Notes in Mathematics}
\vol{2096}
\publ{Springer, 2014, Cham}
\endbook


\bibitem{Mbook}
\by{\name{M\"uller}{S.}}
\book{Variational models for microstructure and phase transition}
\publ{in "Calculus of Variations and Geometric Evolution Problems" (eds. S. Hildebrandt and M. Struwe), Springer-Verlag, (1999), 85-210}
\endbook

\bibitem{Ke}
\by{\name{Kechris}{A. S.}}
\book{Classical descriptive set theory}
\publ{Springer-Verlag, New York, 1995}
\endbook


\bibitem{T}
\by{\name{Toupin}{R.A.}}
\paper{Elastic materials with couple-stresses}
\jour{Arch. Rational Mech. Anal.}
\vol{11}
\pages{385--414}
\yr{1962}
\endpaper




\bibitem{T2}
\by{\name{Toupin}{R.A.}}
\paper{Theories of elasticity with couple-stress}
\jour{Arch. Rational Mech. Anal.}
\vol{17}
\pages{85--112}
\yr{1964}
\endpaper




\end{thebibliography}
\end{document}